\documentclass{article}
\usepackage{amsmath,amsthm,amsopn,amstext,amscd,amsfonts,amssymb}
\usepackage{natbib}

\def \tr {\mathop{\rm tr}\nolimits}

\def \Vol {\mathop{\rm Vol}\nolimits}

\def \etr {\mathop{\rm etr}\nolimits}

\renewenvironment{abstract}
                 {\vspace{6pt}
                  \begin{center}
                  \begin{minipage}{5in}
                  \centerline{\textbf{Abstract}}
                  \noindent\ignorespaces
                 }
                 {\end{minipage}\end{center}}
\newtheorem{thm}{\textbf{Theorem}}[section]

\newtheorem{lem}{\textbf{Lemma}}[section]
\theoremstyle{definition}

\title{\Large \textbf{An identity of Jack polynomials}}
\author{
  \textbf{Jos\'e A. D\'{\i}az-Garc\'{\i}a} \thanks{Corresponding author\newline
    }\\
  {\normalsize Department of Statistics and Computation} \\
  {\normalsize 25350 Buenavista, Saltillo, Coahuila, Mexico} \\
  {\normalsize E-mail: jadiaz@uaaan.mx} \\[2ex]
  \textbf{Ram\'on Guti\'errez J\'aimez} \\
  {\normalsize Department of Statistics and O.R.} \\
  {\normalsize University of Granada} \\
  {\normalsize Granada 18071, Spain}\\
  {\normalsize E-mail: rgjaimez@ugr.es}\\
}
\date{}
\begin{document}
\maketitle

\begin{abstract}
In this work it is propose an alterative proof of one of basic properties of the
zonal polynomials. This identity is generalised for the Jack polynomials.
\end{abstract}

\section{Introduction}\label{sec1}

Many results in multivariate distribution theory have been obtained using zonal and
invariant polynomials. Moreover, these results, in their final version, have been
derived in very compact form, using hypergeometric functions with one or two matrix
arguments, see \citet{c:63}, \citet{j:64}, \citet{chd:79}, \citet{da:80} and
\citet{m:82}, among many others.

Many of these results obtained in the real case have also been studied with respect
to complex, quaternion and octonion cases, see \citet{j:64}, \citet{lx:09} and
\citet{f:09}, and although many properties of real and complex zonal polynomials have
been extended to the quaternion and octonion cases, many others remain unstudied.

In this paper, we are interested in particular in the basic property of real zonal
polynomials, examined in \citet[Theorem 5, eq. (27)]{j:61} (see also \citet[eq.
(22)]{j:64}), and proved by \citet{j:61}, in terms of group representation theory.
This property plays a fundamental role in the context of matrix multivariate
elliptical distributions and specifically in that of related noncentral matrix
multivariate distributions, such as generalised noncentral Wishart and beta
distributions, and also in the context of generalised shape theory, see
\citet{dggf:04b}, \citet{dggj06} and \citet{Caro2009}.

Section \ref{sec2} proposes an alternative proof of one of the basic properties of
zonal polynomials established by \citet[Theorem 5, eq. (27)]{j:61} (see also
\citet[eq. (22)]{j:64}). The proof is given in terms of the results in \citet{h:55}
and \citet{c:63}, and as the main result, this property is generalised for real
normed division algebras.

\section{Main result}\label{sec2}

A detailed discussion of real normed division algebras may be found in \citet{b:02}
and \citet{gr:87}, and of Jack polynomials and hypergeometric functions in
\citet{S:97}, \citet{gr:87} and \citet{KE:06}. For convenience, we shall introduce
some notations, although in general we adhere to standard notations.

There are exactly four real finite-dimensional normed division algebras: real
numbers, complex numbers, quaternions and octonions, these being denoted generically
as $\mathfrak{F}$, see \citet{b:02}. All division algebras have a real dimension of
$1, 2, 4$ or $8$, respectively, whose dimension is denoted by $\beta$, see
\citet[Theorems 1, 2 and 3]{b:02}.

Let ${\mathcal L}^{\beta}_{m,n}$ be the linear space of all $n \times m$ matrices of
rank $m \leq n$ over $\mathfrak{F}$ with $m$ distinct positive singular values, where
$\mathfrak{F}$ denotes a \emph{real finite-dimensional normed division algebra}. Let
$\mathfrak{F}^{n \times m}$ be the set of all $n \times m$ matrices over
$\mathfrak{F}$, and let $\mathbf{A} \in \mathfrak{F}^{n \times m}$. Then
$\mathbf{A}^{*} = \overline{\mathbf{A}}^{T}$ denotes the usual conjugate transpose.

The set of matrices $\mathbf{H}_{1} \in \mathfrak{F}^{n \times m}$ such that
$\mathbf{H}_{1}^{*}\mathbf{H}_{1} = \mathbf{I}_{m}$ is a manifold denoted ${\mathcal
V}_{m,n}^{\beta}$, termed the \emph{Stiefel manifold}. In particular, ${\mathcal
V}_{m,m}^{\beta}$, is the maximal compact subgroup $\mathfrak{U}^{\beta}(m)$ of
${\mathcal L}^{\beta}_{m,m}$ and consists of all matrices $\mathbf{H} \in
\mathfrak{F}^{m \times m}$ such that $\mathbf{H}^{*}\mathbf{H} = \mathbf{I}_{m}$. If
$\mathbf{H}_{1} \in \mathcal{V}_{m,n}^{\beta}$ then
$$
  (\mathbf{H}^{*}_{1}d\mathbf{H}_{1}) = \bigwedge_{i=1}^{n} \bigwedge_{j =i+1}^{m}
  \mathbf{h}_{j}^{*}d\mathbf{h}_{i}.
$$
where $\mathbf{H} = (\mathbf{H}_{1}|\mathbf{H}_{2}) = (\mathbf{h}_{1}, \dots,
\mathbf{h}_{m}|\mathbf{h}_{m+1}, \dots, \mathbf{h}_{n}) \in \mathfrak{U}^{\beta}(m)$.
 The surface area or volume of the Stiefel manifold
$\mathcal{V}^{\beta}_{m,n}$ is
\begin{equation}\label{vol}
    \Vol(\mathcal{V}^{\beta}_{m,n}) = \int_{\mathbf{H}_{1} \in
  \mathcal{V}^{\beta}_{m,n}} (\mathbf{H}^{*}_{1}d\mathbf{H}_{1}) =
  \frac{2^{m}\pi^{mn\beta/2}}{\Gamma^{\beta}_{m}[n\beta/2]},
\end{equation}
where $\Gamma^{\beta}_{m}[a]$ denotes the multivariate Gamma function for the space
of hermitian matrices, see \citet{gr:87}.

Let $C_{\kappa}^{\beta}(\mathbf{B})$ be the Jack polynomials of $\mathbf{B} =
\mathbf{B}^{*}$, corresponding to the partition $\kappa=(k_{1},\ldots k_{m})$ of $k$,
$k_{1} \geq \cdots \geq k_{m} \geq 0$ with $\sum_{i=1}^{m}k_{i}=k$, see \citet{S:97}
and \citet{KE:06}. In addition,
$$
  {}_{p}F_{q}^{\beta}(a_{1}, \dots,a_{p};b_{1}, \dots,b_{q}; \mathbf{B}) = \sum_{k=0}^{\infty}\sum_{\kappa}
  \frac{[a_{1}]_{\kappa}^{\beta}, \dots, [a_{p}]_{\kappa}^{\beta}}{[b_{1}]_{\kappa}^{\beta}, \dots, [b_{p}]_{\kappa}}^{\beta}
  \frac{C_{\kappa}^{\beta}(\mathbf{B})}{k!},
$$
defines the hypergeometric function with one matrix argument on the space of
hermitian matrices, where $[a]^{\beta}_{\kappa}$ denotes the generalised Pochhammer
symbol of weight $\kappa$, defined as
$$
  [a]^{\beta}_{\kappa} = \prod_{i=1}^{m}(a-(i-1)\beta/2)_{k_{1}}
$$
where $ \Re(a) > (m-1)\beta/2 - k_{m}$ and $(a)_{i} = a(a+1)\cdots(a+i-1)$, see
\citet{gr:87}, \citet{KE:06} and \citet{dg:09}.

Now, we clarify an apparent discrepancy between the results obtained by the different
approaches. From \citet[Lemma 9.5.3, p. 397]{m:82}, it is easy to see that equality
(3.5'), proved via a Laplace transform by \citet[p. 494]{h:55}, and equality (27) in
\citet{j:64}, proved via group representation theory by \citet[Theorem 5]{j:61},
coincide.

Then, from \citet[eq. (27)]{j:61} (see also \citet[eq. (22)]{j:64}) we have the
following.

\begin{lem}\label{lemj}
If $\mathbf{X} \in \mathfrak{L}^{1}_{n,m}$, then
\begin{equation}\label{main}
    \int_{\mathbf{H}_{1} \in \mathcal{V}^{1}_{m,n}} (\tr(\mathbf{XH}_{1}))^{2k} (d\mathbf{H}_{1})=
  \sum_{\kappa}\frac{\left(
  \frac{1}{2}\right)_{k}}{[n/2]^{1}_{\kappa}}C^{1}_{\kappa}(\mathbf{XX}^{*}).
\end{equation}
\end{lem}
\begin{proof} From \citet[eq. (3.5'), p. 494]{h:55}, and expanding in series of powers
\begin{eqnarray*}
  {}_{0}F_{1}^{1}(n/2, \mathbf{XX}^{*}/4) &=& \displaystyle \int_{\mathbf{H}_{1} \in
        \mathcal{V}_{m,n}^{1}}  \etr \{\mathbf{XH}_{1}\}(d\mathbf{H}_{1}) \\
    &=& \displaystyle \sum_{k=0}^{\infty}\frac{1}{k!} \int_{\mathbf{H} \in
    \mathcal{V}^{1}_{m,n}} (\tr(\mathbf{XH}_{1}))^{k} (d\mathbf{H}_{1}).
\end{eqnarray*}
Recalling that if one or more parts $k_{1}, \dots, k_{m}$ of partition $k$ is odd,
then
$$
  \int_{\mathbf{H} \in
    \mathcal{V}^{1}_{m,n}} (\tr(\mathbf{XH}_{1}))^{k} (d\mathbf{H}_{1}) = 0,
$$
see \citet{j:61a} and \citet{j:64}. Therefore
\begin{equation}\label{le}
    {}_{0}F_{1}^{1}(n/2, \mathbf{XX}^{*}/4) = \sum_{k=0}^{\infty}\frac{1}{(2k)!} \int_{\mathbf{H} \in
    \mathcal{V}^{1}_{m,n}} (\tr(\mathbf{XH}_{1}))^{2k} (d\mathbf{H}_{1}).
\end{equation}
Now, by the definition of hypergeometric functions with one matrix argument in terms
of zonal polynomials, we have (see \citet{c:63}),
\begin{equation}\label{lh}
    {}_{0}F_{1}^{1}(n/2, \mathbf{XX}^{*}/4) = \sum_{k=0}^{\infty} \sum_{\kappa}
    \frac{1}{[n/2]_{\kappa}^{1}} \frac{C_{\kappa}^{1}(\mathbf{XX}^{*}/4)}{k!}.
\end{equation}
Then, equalling and comparing term-by-term the series on the right side of (\ref{le})
and (\ref{lh}) we obtain
$$
  \sum_{\kappa}
    \frac{1}{[n/2]_{\kappa}^{1}} \frac{C_{\kappa}^{1}(\mathbf{XX}^{*}/4)}{k!} = \frac{1}{(2k)!}
    \int_{\mathbf{H} \in \mathcal{V}^{1}_{m,n}} (\tr(\mathbf{XH}_{1}))^{2k}
    (d\mathbf{H}_{1}).
$$
Finally, observing that $4^{k}(1/2)_{k}/(2k)! = 1/k!$ and that
$C_{\kappa}^{1}(a\mathbf{B}) = a^{k}C_{\kappa}^{1}(\mathbf{B})$, the desired result
is obtained. \qed
\end{proof}

Property (\ref{main}) was also proved in an alternative way by \citet[Lemma 1, p.
40]{t:84}, for the real case. Now, under our approach, property (\ref{main}) is
easily extended to the Jack polynomial case for real normed division algebras.

\begin{thm}\label{lemjack}
Let $\mathbf{X} \in \mathfrak{L}^{\beta}_{n,m}$, then
\begin{equation}\label{mainj}
    \int_{\mathbf{H}_{1} \in \mathcal{V}^{\beta}_{m,n}} (\tr(\mathbf{XH}_{1}))^{2k} (d\mathbf{H}_{1})=
  \sum_{\kappa}\frac{\left(
  \frac{1}{2}\right)_{k}}{[\beta n/2]^{\beta}_{\kappa}}C^{\beta}_{\kappa}(\mathbf{XX}^{*}).
\end{equation}
\end{thm}
\begin{proof} Observe that by \citet{gr:87} and \citet{KE:06},
$$
    {}_{0}F_{1}^{\beta}(\beta n/2, \mathbf{XX}^{*}/4) = \sum_{k=0}^{\infty} \sum_{\kappa}
    \frac{1}{[\beta n/2]_{\kappa}^{\beta}}
    \frac{C_{\kappa}^{\beta}(\mathbf{XX}^{*}/4)}{k!},
$$
and by \citet{dg:09},
$$
  {}_{0}F_{1}^{\beta}( \beta n/2, \mathbf{XX}^{*}/4) = \displaystyle \int_{\mathbf{H}_{1} \in
        \mathcal{V}_{m,n}^{\beta}}  \etr \{\mathbf{XH}_{1}\}(d\mathbf{H}_{1}),
$$
whose equality was found by \citet{j:64}, for the complex case, and by \citet{lx:09},
for the quaternion case. Then, the desired result is obtained following the proof of
Lemma \ref{lemj}. \qed
\end{proof}

\section*{Acknowledgements}
This work was partially supported by IDI-Spain, Grants No. FQM2006-2271 and
MTM2008-05785. This paper was written during J. A. D\'{\i}az-Garc\'{\i}a's stay as a
visiting professor at the Department of Statistics and O. R. of the University of
Granada, Spain.

\bibliographystyle{plain}

\end{document}